\documentclass{amsart}
\usepackage{amssymb,latexsym,enumitem,graphics}

\theoremstyle{cute}
\newtheorem{thm}{Theorem}[section]

\newtheorem{lemma}[thm]{Lemma}
\newtheorem{prop}[thm]{Proposition}

\theoremstyle{definition}

\newtheorem{ex}[thm]{Example}

\theoremstyle{remark}

\numberwithin{equation}{section}

\newcommand{\D}{\mathbb{D}}

\begin{document}
\title[
 Weighted Composition--Differentiation Operators]{ Weighted Composition--Differentiation Operator on the Hardy and Weighted Bergman Spaces}
\author[M. Fatehi]{Mahsa Fatehi}

\date{August 12, 2021}

\address{Department of Mathematics\\ Shiraz Branch, Islamic Azad University\\
Shiraz, Iran}
\email{fatehimahsa@yahoo.com}

\

\begin{abstract}
In this paper, we consider the sum of weighted composition operator $C_{\psi_{0},\varphi_{0}}$ and the weighted composition--differentiation operator $D_{\psi_{n},\varphi_{n},n}$  on the Hardy and weighted Bergman spaces. We describe the spectrum of a compact operator $C_{\psi_{0},\varphi_{0}}+D_{\psi_{n},\varphi_{n},n}$ when  the fixed point $w$ of $\varphi_{0}$ and $\varphi_{n}$ is inside the open unit disk and $\psi_{n}$ has a zero at $w$ of order at least $n$. Also the lower estimate and the upper estimate on the norm of a weighted composition--differentiation operator on the Hardy space  $H^{2}$ are obtained. Furthermore,  we determine the norm of a  composition--differentiation operator $D_{\varphi,n}$, acting on the Hardy space $H^{2}$, in the case where $\varphi(z)=bz$ for some complex number $b$ that $|b|<1$.
\end{abstract}
\subjclass[2010]{47B38 (Primary), 30H10, 30H20, 47B33, 47E99}
\keywords{Weighted composition operator, Differentiation operator, Spectrum, Norm, Hardy space, Weighted Bergman space.}

\maketitle
\thispagestyle{empty}

\section{Preliminaries}
Let $\mathbb{D}$ be the open unit disk in the complex plane.
The
\textit{Hardy space} $H^{2}$ is the set of all analytic functions $f$ on $\D$ such that
$$
\|f\|_=\left(\sup_{0<r<1}\frac{1}{2\pi}
\int_{0}^{2\pi}\bigl|f\bigl(re^{i\theta}\bigr)\bigr|^{2}d\theta\right)^{1/2}
<\infty\text{.}
$$
For
$-1<\alpha<\infty$, the \textit{weighted Bergman space}
$A^{2}_{\alpha}$  is the space of all
analytic functions $f$ on $\mathbb{D}$ so that
$$\|f\|=\left(\int_{\D}\left|f(z)\right|^{2}(\alpha+1)(1-|z|^{2})^{\alpha}dA(z)\right)^{1/2}<\infty\text{,}$$
where $dA$ is the normalized area measure on $\mathbb{D}$.
The case when $\alpha = 0$, usually denoted $A^{2}$, is called the (unweighted) Bergman space.
Throughout this paper, we will write $\mathcal{H}_{\alpha}$ to denote the Hardy space $H^{2}$ for $\alpha=-1$ or the weighted Bergman space $A_{\alpha}^{2}$ for $\alpha>-1$.\par
The weighted Bergman spaces and the Hardy space are \textit{reproducing kernel} Hilbert spaces. For every $w \in \mathbb{D}$ and each non-negative integer $n $, let $K_{w,\alpha}^{[n]}$ denote the unique function in $\mathcal{H}_{\alpha}$ that $ \langle f, K_{w,\alpha}^{[n]}\rangle=f^{(n)}(w)$ for each $f \in \mathcal{H}_{\alpha}$, where $f^{(n)}$ is the $n$th derivative of $f$ (note that $f^{(0)}=f$); for convenience, we
use the notation $K_{w,\alpha}$ when $n=0$. The function $K_{w,\alpha}^{[n]}$ is called   the
\textit{reproducing kernel function}. The
reproducing kernel functions for evaluation at $w$ are given by $K_{w,\alpha}(z)=1/(1-\overline{w}z)^{\alpha+2}$ and
\begin{eqnarray*}
K_{w,\alpha}^{[n]}(z)= \frac{(\alpha+2)...(\alpha+n+1) z^n}{(1-\overline{w}z)^{n+\alpha+2}}
\end{eqnarray*}
for $z,w\in \mathbb{D}$  and  $n> 1$. \par

For an operator $T$ on $\mathcal{H}_{\alpha}$, we write $\|T\|_{\alpha}$ to denote the norm of $T$ acting on $\mathcal{H}_{\alpha}$. Through this paper, the spectrum  of  $T$ and the point spectrum of  $T$ and the  spectral radius of $T$ are denoted by $\sigma_{\alpha}(T)$, $\sigma_{p,\alpha}(T)$, and $r_{\alpha}(T)$, respectively.\par
We write $H^{\infty}$ to denote the space of all bounded analytic functions on $\D$, with $\|f\|_{\infty}=\sup\{|f(z)|:z\in\D\}$.\par

We say that an operator $T$ on a Hilbert space $H$ is \textit{hyponormal} if $T^{\ast}T-TT^{\ast} \geq 0$, or equivalently if $\|T^{\ast}f\| \leq \|Tf\|$ for all $f \in H$. Moreover,  the operator $T$ is said to be \textit{ cohyponormal} if $T^{\ast}$ is hyponormal.
Let $P$ denote the
projection of $L^{2}(\partial \mathbb{D})$ onto $H^{2}$. For each $b \in {L^{2}(\partial \mathbb{D})}$, we define
the \textit{Toeplitz  operator} $T_{b}$ on $H^{2}$ by
$T_{b}(f)=P(bf)$.
 For $\varphi$ an analytic self-map of $\mathbb{D}$, let $C_{\varphi}$ be the \textit{composition operator}
 such that $C_{\varphi}(f) = f \circ \varphi$ for any $f \in \mathcal{H}_{\alpha}$. The composition operator $C_{\varphi}$  acts boundedly for every $\varphi$, with
\[
\left(\frac{1}{1-|\varphi(0)|^{2}}\right)^{(\alpha+2)/2}\leq\|C_{\varphi}\|_{\alpha}\leq\left(\frac{1+|\varphi(0)|}{1-|\varphi(0)|}\right)^{(\alpha+2)/2}\text{.}
\]
(See \cite[Corollary 3.7]{cm} and \cite[Lemma 2.3]{richman}.) Let $\psi$ be an
 analytic function on $\mathbb{D}$ and $\varphi$ be an analytic self-map of $\mathbb{D}$. The \textit{weighted
composition operator} $C_{\psi , \varphi}$ is defined by $C_{\psi ,
\varphi}(f)=\psi \cdot (f\circ \varphi)$ for  $f \in \mathcal{H}_{\alpha}$.\par

Although for each positive integer $n$, the \textit{differentiation operator} $D_{n}(f)=f^{(n)}$ is unbounded on $ \mathcal{H}_{\alpha}$ (note that $\lim_{m \rightarrow \infty}\|D_{n}(z^{m})\|/\|z^{m}\|=\infty$), there are some analytic maps $\varphi:\mathbb{D}\rightarrow \mathbb{D}$ such that the operator $C_{\varphi}D_{n}$ is bounded. The bounded and compact  operators $C_{\varphi}D_{n}$  on $\mathcal{H}_{\alpha}$ were determined in \cite{hp}, \cite{moradi}, \cite{ohno} and \cite{s4}. Recently the author and Hammond \cite{fh}  obtained the adjoint, norm, and spectrum of some operators $C_{\varphi}D_{1}$ on the Hardy space.  For an analytic self-map $\varphi$ of $\mathbb{D}$ and a positive integer $n$, the \textit{composition--differentiation operator} on $\mathcal{H}_{\alpha}$ is defined by the rule $D_{\varphi,n}(f)=f^{(n)}\circ \varphi$; for convenience, we
use the notation $D_{\varphi}$ when $n=1$. For  an analytic function $\psi$   on  $\mathbb{D}$, the \textit{weighted composition--differentiation operator} $D_{\psi,\varphi,n}$ on $\mathcal{H}_{\alpha}$ is defined
$$D_{\psi,\varphi,n}f(z)=\psi(z)f^{(n)}(\varphi(z))\text{.}$$
Some properties of  weighted composition--differentiation operators were considered in \cite{s2} and \cite{hw}.\par

In this paper, we determine  the spectrum of a compact operator $C_{\psi_{0},\varphi_{0}}+D_{\psi_{n},\varphi_{n},n}$ when  the fixed point $w$ of $\varphi_{0}$ and $\varphi_{n}$ is inside the open unit disk and the function $\psi_{n}$ has a zero at $w$ of order at least $n$ (Theorem \ref{4theorem}).  The spectral radius of a class of compact weighted composition--differentiation operators is obtained (Theorem \ref{20theorem}). In addition, we find the lower estimate and the upper estimate for $\|D_{\psi,\varphi,n}\|_{-1}$ (Propositions \ref{5theorem} and \ref{8prop}).  Moreover, the norm of a  composition--differentiation operator $D_{\varphi,n}$, acting on the Hardy space $H^{2}$, is determined in the case where $\varphi(z)=bz$ for some complex number $b$ that $|b|<1$ (Theorem \ref{7theorem}).

\section{Spectral Properties}

To find the spectrum of $C_{\psi_{0},\varphi_{0}}+D_{\psi_{n},\varphi_{n},n}$ we need to obtain an invariant subspace of $\big(C_{\psi_{0},\varphi_{0}}+D_{\psi_{n},\varphi_{n},n}\big)^{\ast}$. To do this, we consider the action of the adjoint of the operator $C_{\psi_{0},\varphi_{0}}+D_{\psi_{n},\varphi_{n},n}$ on the reproducing kernel functions.

\begin{lemma}\label{lemma2}
Let $m$ be a non-negative integer. Suppose that $C_{\psi_{0},\varphi_{0}}$ and $D_{\psi_{n},\varphi_{n},n}$ are bounded operators on  $\mathcal{H}_{\alpha}$ and the fixed point $w$ of $\varphi_{0}$ and $\varphi_{n}$ is inside the open unit disk. Assume that the function $\psi_{n}$ has a zero at $w$ of order at least $n$.
\begin{enumerate}[label={(\roman*)}]
\item If $m> n$, then
\begin{align*}
\big(C_{\psi_{0},\varphi_{0}}+ D_{\psi_{n},\varphi_{n},n}\big)^{\ast}K_{w,\alpha}^{[m]}&=\sum_{j=0}^{m-1}\overline{\alpha_{j}(w)}K_{w,\alpha}^{[j]}+
\sum_{i=n}^{m-1}
\overline{\beta_{i-n}(w)}K_{w,\alpha}^{[i]}\\
&+\bigg(\overline{\psi_{0}(w)\big(\varphi^{\prime}_{0}(w)\big)^{m}}+\binom{m}{n}\overline{\psi_{n}^{(n)}(w)
\big(\varphi^{\prime}_{n}(w)\big)
^{m-n}}\bigg)K_{w,\alpha}^{[m]}\text{;}
 \end{align*}
 \item if $m= n$, then
 \begin{align*}
 \big(C_{\psi_{0},\varphi_{0}}+ D_{\psi_{n},\varphi_{n},n}\big)^{\ast}K_{w,\alpha}^{[m]}=\sum_{j=0}^{n-1}\overline{\alpha_{j}(w)}K_{w,\alpha}^{[j]}
 +\bigg(\overline{\psi_{0}(w)\big(\varphi^{\prime}_{0}(w)\big)^{n}}+\overline{\psi_{n}^{(n)}(w)}\bigg)K_{w,\alpha}^{[n]}\text{;}
 \end{align*}
 \item if $m<n$, then
 \begin{align*}
 \big(C_{\psi_{0},\varphi_{0}}+ D_{\psi_{n},\varphi_{n},n}\big)^{\ast}K_{w,\alpha}^{[m]}=\sum_{j=0}^{m-1}\overline{\alpha_{j}(w)}K_{w,\alpha}^{[j]}
 +\overline{\psi_{0}(w)\big(\varphi^{\prime}_{0}(w)\big)^{m}}K_{w,\alpha}^{[m]}\text{,}
 \end{align*}
 \end{enumerate}
where the functions $\alpha_{j}$'s and $\beta_{j}$'s  consist of some products of the derivatives of $\psi_{0}$ and $\varphi_{0}$ and some products of the derivatives of $\psi_{n}$ and $\varphi_{n}$, respectively.
\end{lemma}

\begin{proof}
Let $f$ be an arbitrary function in $\mathcal{H}_{\alpha}$. For each non-negative integer $m$, we have
\begin{align*}
\big\langle f,C_{\psi_{0},\varphi_{0}}^{\ast}K_{w,\alpha}^{[m]}\big\rangle&=\big\langle C_{\psi_{0},\varphi_{0}}f,K_{w,\alpha}^{[m]} \big\rangle=\sum_{j=0}^{m}\binom{m}{j}\psi_{0}^{(m-j)}(w)\big(f\circ\varphi_{0}\big)^{(j)}(w)\\
&=\bigl\langle f, \sum_{j=0}^{m-1}\overline{\alpha_{j}(w)}K_{w,\alpha}^{[j]}+
\overline{\psi_{0}(w)\big(\varphi^{\prime}_{0}(w)\big)^{m}}K_{w,\alpha}^{[m]}\bigr\rangle\text{.}
  \end{align*}
Since $f$ is an arbitrary function in $\mathcal{H}_{\alpha}$, we conclude that
\begin{equation}\label{e1}
C_{\psi_{0},\varphi_{0}}^{\ast}K_{w,\alpha}^{[m]}=\sum_{j=0}^{m-1}\overline{\alpha_{j}(w)}K_{w,\alpha}^{[j]}+
\overline{\psi_{0}(w)\big(\varphi'_{0}(w)\big)^{m}}K_{w,\alpha}^{[m]}\text{.}
\end{equation}\par
  Let $m<n$. Since $\psi_{n}$ has a zero at $w$ of order at least $n$, we have
 \begin{align*}
  \big\langle f,D_{\psi_{n},\varphi_{n},n}^{\ast}K_{w,\alpha}^{[m]}\big\rangle&=\big(\psi_{n}\cdot\big(f^{(n)}\circ \varphi_{n}\big)\big)^{(m)}(w)\\
  &=\sum_{i=0}^{m}\binom{m}{i}\psi_{n}^{(m-i)}(w)\big(f^{(n)}\circ \varphi_{n}\big)^{(i)}(w)\\
  &=0\text{.}
   \end{align*}
  It shows that $D_{\psi_{n},\varphi_{n},n}^{\ast}K_{w,\alpha}^{[m]}=0$.\par
   Now assume that $m\geq n$. We obtain
   \begin{eqnarray} \label{m4}
   \bigl\langle f,   D_{\psi_{n},\varphi_{n},n}^{\ast}K_{w,\alpha}^{[m]}\bigr\rangle&=&\sum_{i=0}^{m}\binom{m}{i}\psi_{n}^{(m-i)}(w)
   \big(f^{(n)}\circ\varphi_{n}\big)^{(i)}(w)\nonumber \\
   &=&\sum_{i=0}^{m-n}\binom{m}{i}\psi_{n}^{(m-i)}(w)
   \big(f^{(n)}\circ\varphi_{n}\big)^{(i)}(w)\nonumber \\
   &+&\sum_{i=m-n+1}^{m}\binom{m}{i}\psi_{n}^{(m-i)}(w)
   \big(f^{(n)}\circ\varphi_{n}\big)^{(i)}(w)\nonumber \\
   &=&\sum_{i=0}^{m-n}\binom{m}{i}\psi_{n}^{(m-i)}(w)
   \big(f^{(n)}\circ\varphi_{n}\big)^{(i)}(w)\text{.}
   \end{eqnarray}
   If $m>n$, then  by (\ref{m4}), we get
   \begin{align*}
   \bigl\langle f,   D_{\psi_{n},\varphi_{n},n}^{\ast}K_{w,\alpha}^{[m]}\bigr\rangle&= \bigl\langle f,\sum_{i=0}^{m-n-1}\overline{\beta_{i}(w)}K_{w,\alpha}^{[i+n]}+\binom{m}{m-n}\overline{\psi_{n}^{(n)}(w)
   \big(\varphi^{\prime}_{n}(w)\big)^{m-n}}K_{w,\alpha}^{[m]}\bigr\rangle\text{,}
   \end{align*}
   so
   $$D_{\psi_{n},\varphi_{n},n}^{\ast}K_{w,\alpha}^{[m]}=\sum_{i=n}^{m-1}\overline{\beta_{i-n}(w)}K_{w,\alpha}^{[i]}+
   \binom{m}{n}\overline{\psi_{n}^{(n)}(w)
   \big(\varphi^{\prime}_{n}(w)\big)^{m-n}}K_{w,\alpha}^{[m]}\text{.}$$
   If $m=n$, then by (\ref{m4}), we see that
    \begin{align*}
   \bigl\langle f,   D_{\psi_{n},\varphi_{n},n}^{\ast}K_{w,\alpha}^{[m]}\bigr\rangle&=\psi_{n}^{(n)}(w)
   f^{(n)}(w)=\bigl\langle f,\overline{\psi_{n}^{(n)}(w)}K_{w,\alpha}^{[n]}\bigr\rangle \text{.}
  \end{align*}
   Hence the result follows.
\end{proof}
In the next proposition, we identify all possible eigenvalues of $C_{\psi_{0},\varphi_{0}}+D_{\psi_{n},\varphi_{n},n}$.

\begin{prop}\label{1prop}
 Suppose that $C_{\psi_{0},\varphi_{0}}$ and $D_{\psi_{n},\varphi_{n},n}$ are bounded operators on  $\mathcal{H}_{\alpha}$ and the fixed point $w$ of $\varphi_{0}$ and $\varphi_{n}$ is inside the open unit disk.
If the function $\psi_{n}$ has a zero at $w$ of order at least $n$, then
$$\bigg\{\psi_{0}(w),\psi_{0}(w)\big(\varphi^{\prime}_{0}(w)\big)^{l}:l \in \mathbb{N}_{<n}  \bigg\}\bigcup  \bigg\{ \psi_{0}(w)\big(\varphi^{\prime}_{0}(w)\big)^{l}+\binom{l}{n}\psi_{n}^{(n)}(w)\big(\varphi^{\prime}_{n}(w)
\big)^{l-n}: l \in \mathbb{N}_{\geq n}\bigg\}$$
contains the point spectrum of
$C_{\psi_{0},\varphi_{0}}+D_{\psi_{n},\varphi_{n},n}$.
\end{prop}

\begin{proof}
Let $\lambda$ be an arbitrary eigenvalue for $C_{\psi_{0},\varphi_{0}}+D_{\psi_{n},\varphi_{n},n}$ with corresponding eigenvector $f$. Note that
\begin{equation}\label{e2}
\lambda f(z)=\psi_{0}(z)f\big(\varphi_{0}(z)\big)+\psi_{n}(z)f^{(n)}\big(\varphi_{n}(z)\big)
\end{equation}
for each $z \in \mathbb{D}$. If $f(w)\neq 0$, then $\lambda=\psi_{0}(w)$. Let $f$ have a zero at $w$ of order $l\geq 1$. Differentiate (\ref{e2}) $l$ times and evaluate it at the point $z=w$ to obtain
\begin{eqnarray}\label{e3}
\lambda f^{(l)}(w)&=&\sum_{j=0}^{l}\binom{l}{j}\psi_{0}^{(l-j)}(w)(f\circ \varphi_{0})^{(j)}(w)\nonumber \\
&+&\sum_{j=0}^{l}\binom{l}{j}\psi_{n}^{(l-j)}(w)(f^{(n)}\circ \varphi_{n})^{(j)}(w)\text{.}
\end{eqnarray}\par
First assume that $l<n$. Since $\psi_{n}$ has a zero at $w$ of order at least $n$, we have $\lambda=\psi_{0}(w)\big(\varphi_{0}^{\prime}(w)\big)^{l}$ by (\ref{e3}).\par
Now assume that $l\geq n$. Then $\psi_{n}^{(l-j)}(w)=0$ for each $j>l-n$. Hence (\ref{e3}) implies that
$$\lambda f^{(l)}(w)=\sum_{j=0}^{l}\binom{l}{j}\psi_{0}^{(l-j)}(w)\big(f\circ \varphi_{0}\big)^{(j)}(w)+\sum_{j=0}^{l-n}\binom{l}{j}\psi_{n}^{(l-j)}(w)\big(f^{(n)}\circ \varphi_{n}\big)^{(j)}(w)$$
and so
$$\lambda f^{(l)}(w)=\psi_{0}(w)f^{(l)}(w)\big( \varphi_{0}^{\prime}(w)\big)^{l}+\binom{l}{l-n}\psi_{n}^{(n)}(w)f^{(l)}(w) \big(\varphi_{n}^{\prime}(w)\big)^{l-n}\text{.}$$
(Note that in case of $\varphi^{\prime}_{n}(w)=0$ and $l=n$, we set $\big(\varphi_{n}^{\prime}(w)\big)^{l-n}=1$.)
Therefore, in this case, any eigenvalue must have the form
$$\psi_{0}(w)\big( \varphi_{0}^{\prime}(w)\big)^{l}+\binom{l}{n}\psi_{n}^{(n)}(w) \big(\varphi_{n}^{\prime}(w)\big)^{l-n}$$
for a natural number $l$ with $l\geq n$.
\end{proof}

\begin{prop}\label{2prop}
Suppose that the hypotheses of Proposition \ref{1prop} hold. Then the point spectrum of  $\big(C_{\psi_{0},\varphi_{0}}+ D_{\psi_{n},\varphi_{n},n}\big)^{\ast}$ contains
$$\bigg\{\overline{\psi_{0}(w)},\overline{\psi_{0}(w)\big(\varphi^{\prime}_{0}(w)\big)^{l}}:l \in \mathbb{N}_{<n}  \bigg\}\bigcup  \bigg\{ \overline{\psi_{0}(w)\big(\varphi^{\prime}_{0}(w)\big)^{l}}+\binom{l}{n}\overline{\psi_{n}^{(n)}(w)
\big(\varphi^{\prime}_{n}(w)
\big)^{l-n}}
: l \in \mathbb{N}_{\geq n}\bigg\}\text{.}$$
\end{prop}

\begin{proof}
Let $l$ be a positive integer with $l\geq n$ and $K_{l}$ denote the span of $\{K_{w,\alpha},K_{w,\alpha}^{[1]},...,K_{w,\alpha}^{[l]}\}$. Note that this spanning set is linearly independent and so is a basis. Let $A_{l}$ be the matrix of the operator $\big(C_{\psi_{0},\varphi_{0}}+ D_{\psi_{n},\varphi_{n},n}\big)^{\ast}$ restricted to $K_{l}$ with respect to this basis. We infer from  Lemma
\ref{lemma2} that
$$ A_{l}=\left[\begin{array}{*5c}
B_{n}&\ast &...&\ast\\
 0&\overline{\psi_{0}(w)\big(\varphi^{\prime}_{0}(w)\big)^{n}+\psi_{n}^{(n)}(w)}&...&\ast\\
 0&0&...&\ast\\
 \vdots&\vdots&\ddots&\vdots\\
 0&0&...&\overline{\psi_{0}(w)\big(\varphi^{\prime}_{0}(w)\big)^{l}}+\binom{l}{n}\overline{\psi_{n}^{(n)}(w)\big(\varphi^{\prime}_{n}(w)
\big)^{l-n}}
 \end{array}\right]\text{,}$$\\
 where $B_{n}$ is an $n\times n$ upper triangular matrix that its  main diagonal entries are $\overline{\psi_{0}(w)}, \overline{\psi_{0}(w)\varphi^{\prime}_{0}(w)},...,\overline{\psi_{0}(w)
 \big(\varphi^{\prime}_{0}(w)\big)^{n-1}} $. Then $A_{l}$ is an upper triangular matrix too. Since the subspace $K_{l}$ is finite dimensional, it is closed and so the space $\mathcal{H}_{\alpha}$ can be decomposed as
 $$\mathcal{H}_{\alpha}=K_{l}\oplus K_{l}^{\perp}\text{.}$$
 Then the block matrix of $\big(C_{\psi_{0},\varphi_{0}}+ D_{\psi_{n},\varphi_{n},n}\big)^{\ast}$ with respect to the above decomposition must be of the form
 $$ \left[\begin{array}{*2c}
A_{l}&C_{l}\\
 0&E{l}
 \end{array}\right]\text{}$$\\
(note that  $K_{l}$ is invariant under $\big(C_{\psi_{0},\varphi_{0}}+ D_{\psi_{n},\varphi_{n},n}\big)^{\ast}$ by Lemma \ref{lemma2} and so the lower left corner of the above matrix is $0$). Since the spectrum of $\big(C_{\psi_{0},\varphi_{0}}+ D_{\psi_{n},\varphi_{n},n}\big)^{\ast}$ is the union of the spectrum of $A_{l}$ and the spectrum of $E_{l}$ (see \cite[p. 270]{cm}), we conclude that the union of
$\bigg\{\overline{\psi_{0}(w)},\overline{\psi_{0}(w)\big(\varphi^{\prime}_{0}(w)\big)^{t}}:t \in \mathbb{N} ~\text{and}~ t<n  \bigg\}$ and   $\bigg\{ \overline{\psi_{0}(w)\big(\varphi^{\prime}_{0}(w)\big)^{t}}+\binom{t}{n}\overline{\psi_{n}^{(n)}(w)
\big(\varphi^{\prime}_{n}(w)
\big)^{t-n}}
: t \in \mathbb{N}~\text{and} ~n\leq t \leq l\bigg\}$ is the subset of  $\sigma_{p,\alpha}\big(\big(C_{\psi_{0},\varphi_{0}}+ D_{\psi_{n},\varphi_{n},n}\big)^{\ast}\big)$. Since $l$ is arbitrary, the result follows.
\end{proof}


 In the following theorem, we characterize the spectrum of an operator  $D_{\psi,\varphi,n}$ under the conditions of Proposition \ref{1prop}. The spectrum of an operator $D_{\psi,\varphi,n}$ which was obtained in \cite[Theorem 3.1]{hw} is an example for Theorem  \ref{4theorem}.

\begin{thm}\label{4theorem}
Suppose that the hypotheses of Proposition \ref{1prop} hold.
If $C_{\psi_{0},\varphi_{0}}+D_{\psi_{n},\varphi_{n},n}$ is compact on $\mathcal{H}_{\alpha}$, then
\begin{align*}
  \sigma_{\alpha}(C_{\psi_{0},\varphi_{0}}+D_{\psi_{n},\varphi_{n},n})&=\{0\}\\
  &\bigcup\bigg\{\psi_{0}(w),\psi_{0}(w)\big(\varphi^{\prime}_{0}(w)\big)^{l}:l \in \mathbb{N}_{<n}  \bigg\}\\
  &\bigcup \bigg\{ \psi_{0}(w)\big(\varphi^{\prime}_{0}(w)\big)^{l}+\binom{l}{n}\psi_{n}^{(n)}(w)\big(\varphi^{\prime}_{n}(w)
\big)^{l-n}
: l \in \mathbb{N}_{\geq n}\bigg\}\text{.}
   \end{align*}
   In particular,  if $\psi_{n}^{(n)}(w)=0$, then the operator $D_{\psi_{n},\varphi_{n},n}$ is quasinilpotent; that is, its spectrum is $\{0\}$.
\end{thm}

In the next theorem, we obtain the spectral radius of a compact operator  $D_{\psi,\varphi,n}$.

\begin{thm}\label{20theorem}
Suppose that $D_{\psi,\varphi,n}$ is a compact operator on  $\mathcal{H}_{\alpha}$. Assume that the fixed point $w$ of $\varphi$ is inside the open unit disk and the function $\psi$ has a zero at $w$ of order $n$.
Then

$$r_{\alpha}\big(D_{\psi,\varphi,n}\big)=\binom{\big\lfloor \frac{n}{1-|\varphi^{\prime}(w)|}\big\rfloor}{n}\big|\psi^{(n)}(w)\big|\big|\varphi^{\prime}(w)\big|^{\big\lfloor \frac{n}{1-|\varphi^{\prime}(w)|}\big\rfloor-n}\text{,}$$
where $\lfloor \cdot \rfloor$ denotes the greatest integer function.
\end{thm}

\begin{proof}
 Theorem \ref{4theorem} implies that $$\sigma_{\alpha}\big(D_{\psi,\varphi,n}\big)=\bigg\{ \binom{l}{n}\psi^{(n)}(w)\big(\varphi^{\prime}(w)
\big)^{l-n}
: l \in \mathbb{N}_{\geq n}\bigg\}
$$
and so
 $$r_{\alpha}\big(D_{\psi,\varphi,n}\big)=\sup \bigg\{ \binom{l}{n}\big|\psi^{(n)}(w)\big|\big|\varphi^{\prime}(w)
\big|^{l-n}
: l \in \mathbb{N}_{\geq n}\bigg\}\text{.}$$
If $\varphi^{\prime}(w)=0$, then $r_{\alpha}\big(D_{\psi,\varphi,n}\big)= \big|\psi^{(n)}(w)\big|$. Now suppose that
$\varphi^{\prime}(w)\neq0$.
Let the function $h(x)=x(x-1)...(x-n+1)\big|\varphi^{\prime}(w)\big|^{x-n}$ on $[n,+\infty)$. Since $\big|\varphi^{\prime}(w)\big| < 1$ (see the Grand Iteration Theorem), we conclude that $\lim_{x \rightarrow \infty}h(x)=0$.
Then $h$ is a bounded function on $[n,+\infty)$ and so it obtains an absolute maximum point. If $h^{\prime}(t)=0$ for some $t\in [n,+\infty)$, then $g(t)=-\ln \big|\varphi^{\prime}(w)\big|$, where $g(x)=\frac{1}{x}+\frac{1}{x-1}+...+\frac{1}{x-n+1}$ for each $x\in [n,+\infty)$. We can easily see that $g^{\prime}$ is strictly decreasing and so the function $h$ has at most one local extremum on $[n,+\infty)$, which must be its absolute maximum (note that if $h^{\prime}(t)\neq 0$ for all $t$, then $h$ has an absolute maximum of $n!$ at $n$).
Therefore, for obtaining $r_{\alpha}\big(D_{\psi,\varphi,n}\big)$, we must find the greatest natural number $l$ such that $l\geq n$ and
$$(l-1)...(l-n)\big|\varphi^{\prime}(w)\big|^{l-n-1}\leq l...(l-n+1)\big|\varphi^{\prime}(w)\big|^{l-n}
$$
or equivalently
$$l\leq \frac{n}{1-\big|\varphi^{\prime}(w)\big|}$$
(note that if  $n!=n(n-1)...1\cdot\big|\varphi^{\prime}(w)\big|^{n-n}>l...(l-n+1)\big|\varphi^{\prime}(w)\big|^{l-n}$ for each $l>n$, then we have $n!>(n+1)!\big|\varphi^{\prime}(w)\big|$. It shows that $n<\frac{n}{1-\big|\varphi^{\prime}(w)\big|}<n+1$ and so $\big\lfloor\frac{n}{1-|\varphi^{\prime}(w)|}\big\rfloor=n$).
Thus the quantity $\binom{l}{n}\big|\varphi^{\prime}(w)
\big|^{l-n}$ is maximized when $l=\bigg\lfloor\frac{n}{1-\big|\varphi^{\prime}(w)\big|}\bigg\rfloor$, so the conclusion follows.
\end{proof}

Suppose that $D_{\varphi}$ is compact on $\mathcal{H}_{\alpha}$. Then $D_{\varphi^{\prime }\circ\varphi,\varphi_{2},2}=D_{\varphi}D_{\varphi}$ is compact. Let $w \in \mathbb{D}$ be a fixed point of $\varphi$. If $\lambda$ is an eigenvalue for $D_{\varphi}$ corresponding to the eigenvector $f$, then  $\lambda^{2}$ is an eigenvector for $D_{\varphi^{\prime}\circ \varphi,\varphi_{2},2}$ corresponding to eigenvector $f$. First suppose that $\varphi^{\prime}(w)=\varphi^{\prime \prime}(w)=0$.   Theorem \ref{4theorem} dictates that $D_{\varphi^{\prime}\circ \varphi,\varphi_{2},2}$  is quasinilpotent and so $D_{\varphi}$ is quasinilpotent.
Now suppose that $\varphi^{ \prime}(w)=0$  and $\varphi^{\prime \prime}(w)\neq0$. Then $D_{\varphi}^{\ast}K_{w,\alpha}^{[2]}=\overline{\varphi^{\prime \prime}(w)}K_{w,\alpha}^{[2]}$ by \cite[Lemma 1]{s2}. Hence $\varphi^{\prime \prime}(w)$ is an eigenvalue for $D_{\varphi}$ and so $D_{\varphi}$ is not quasinilpotent. Moreover, using
Theorem \ref{20theorem} for $D_{\varphi^{\prime}\circ \varphi,\varphi_{2},2}$ shows that $\sigma_{\alpha}\big(D_{\varphi^{\prime}\circ \varphi,\varphi_{2},2}\big)=\{0,\big(\varphi^{\prime \prime}(w)\big)^{2}\}$, so
$$\{0,\varphi^{\prime \prime}(w)\} \subseteq \sigma_{\alpha}(D_{\varphi})\subseteq \{0,\varphi^{\prime \prime}(w), -\varphi^{\prime \prime}(w)\}\text{.}$$\par

Assume that $\varphi\equiv a$, where $a$ is constant  with $|a|<1$ so that $D_{\psi,\varphi,n}$ is bounded on $\mathcal{H}_{\alpha}$. Since $\|\varphi\|_{\infty}< 1$, the operator $D_{\psi,\varphi,n}$ is compact (see \cite{moradi}  and \cite{s4}). The spectra of some  of such operators $D_{\psi,\varphi,n}$ were found in \cite[Theorem 3.2]{hw} and \cite[Theorem 3.3]{hw}, but by the same idea which was stated in the proof of \cite[Theorem 3.2]{hw}, we can easily see that for  these operators, we obtain
$$
\sigma_{\alpha}\big(D_{\psi,\varphi,n}\big):=\begin{cases} \{0\}\cup\big\{\psi^{(n)}(a)\big\}, & \psi^{(n)}(a)\neq 0,\\
 \{0\},& \psi^{(n)}(a)= 0;\end{cases}
$$
moreover, if $\psi^{(n)}(a)\neq 0$, then $\psi$ is an eigenvector for $D_{\psi,\varphi,n}$ with corresponding eigenvalue $\psi^{(n)}(a)$.

\section{Norms}

We begin this section with an example which is a starting point for estimating a lower bound for $\big\|D_{\psi,\varphi,n}\big\|_{-1}$.

\begin{ex}\label{1example}
Suppose that $\varphi(z)=bz^{3}+az^{2}$ with $\frac{1}{2} < |a| < 1$ and $|a|+|b|< 1$. We can see that $\varphi(0)=\varphi^{\prime}(0)=0$ and $\varphi^{\prime \prime}(0)=2a$. By the  paragraph after Theorem \ref{20theorem}, we have $r_{\alpha}(D_{\varphi})=2|a|$ and so
$\big\|D_{\varphi}\big\|_{\alpha}\geq 2|a|>1$. Compare $2|a|$ with the lower bound for $\big\|D_{\varphi}\big\|_{-1}$ which was found in \cite[Proposition 4]{fh} (note that \cite[Proposition 4]{fh} implies that $\big\|D_{\varphi}\big\|_{-1}\geq 1$).
\end{ex}

The preceding example leads to obtain the lower estimate on the norm of  $D_{\psi,\varphi,n}$ on the Hardy space  by  using  the spectrum of a weighted composition--differentiation operator which was obtained in
Proposition \ref{2prop}.

\begin{prop}\label{5theorem}
Suppose that  $D_{\psi,\varphi,n}$ is a bounded operator on $H^{2}$. Assume that the fixed point $w$ of $\varphi$ is inside the open unit disk.
 \begin{enumerate}[label={(\roman*)}]
\item If $\varphi^{\prime}(w) \neq 0$, then
$$\big\|D_{\psi,\varphi,n}\big\|_{-1} \geq \big|\phi^{(n)}(w)\big|\binom{\big\lfloor \frac{n}{1-|\varphi^{\prime}(w)|}\big\rfloor}{n}\big|\varphi^{\prime}(w)\big|^{\big\lfloor \frac{n}{1-|\varphi^{\prime}(w)|}\big\rfloor-n}\text{;}$$
\item if $\varphi^{\prime}(w) = 0$, then
$$\big\|D_{\psi,\varphi,n}\big\|_{-1} \geq \big|\phi ^{(n)}(w)\big| \text{;}$$
\item if $\varphi^{\prime}(w) = 0$, $\psi^{\prime \prime}(w)=0$ and $n=1$, then
$$\big\|D_{\psi,\varphi,1}\big\|_{-1} \geq \max\bigg\{\big|\phi ^{\prime}(w)\big|,\big|\psi(w)\varphi^{\prime \prime}(w)\big|\bigg\}\text{,}$$

\end{enumerate}
where
$$
\phi(z):=\begin{cases} \psi(z), & \psi^{(0)}(w)=...=\psi^{(n-1)}(w)=0,\\
\psi(z)\big(\frac{w-z}{1-\overline{w}z}\big)^{n-m}, & \psi^{(0)}(w)=...=\psi^{(m-1)}(w)=0\text{,}~\psi^{(m)}(w)\neq0 ~\text{and}~ 1\leq m<n \text{,}\\
\psi(z)\big(\frac{w-z}{1-\overline{w}z}\big)^{n}, & \psi(w)\neq 0\text{.}
\end{cases}
$$
\end{prop}

\begin{proof}
First suppose that $\psi^{(0)}(w)=...=\psi^{(n-1)}(w)=0$. Proposition \ref {2prop} and the idea which was used in the proof of Theorem \ref{20theorem} imply that
\begin{equation}\label{e3.1}
\big\|D_{\psi,\varphi,n}\big\|_{-1}\geq \big|\psi^{(n)}(w)\big|\binom{\big\lfloor \frac{n}{1-|\varphi^{\prime}(w)|}\big\rfloor}{n}\big|\varphi^{\prime}(w)\big|^{\big\lfloor \frac{n}{1-|\varphi^{\prime}(w)|}\big\rfloor-n}\text{.}
\end{equation}
(Note that in case of $\varphi^{\prime}(w)=0$, we set  $\big|\varphi^{\prime}(w)\big|^{\big\lfloor \frac{n}{1-|\varphi^{\prime}(w)|}\big\rfloor-n}=1$.)\par
 Now assume that $\psi(z)=(w-z)^{m}g(z)$, where $1\leq m<n$ and $g(w)\neq 0$. Let $\phi(z)=\psi(z)\big(\frac{w-z}{1-\overline{w}z}\big)^{n-m}$.
 Since $T_{\frac{w-z}{1-\overline{w}z}}$ is an isometry on $H^{2}$ and the $n$th derivative
of $\psi(z)\big(\frac{w-z}{1-\overline{w}z}\big)^{n-m}$  at the point $w$ is $\frac{(-1)^{n}n!g(w)}{\big(1-|w|^{2}\big)^{n-m}}$, by replacing $\phi$ with $\psi$ in (\ref{e3.1}), we  obtain
$$\big\|D_{\psi,\varphi,n}\big\|_{-1}=\big\|D_{\phi,\varphi,n}\big\|_{-1} \geq \frac{n!| g(w)|}{\big(1-|w|^{2}\big)^{n-m}}
\binom{\big\lfloor \frac{n}{1-|\varphi^{\prime}(w)|}\big\rfloor}{n}
\big|\varphi^{\prime}(w)\big|^{\big\lfloor \frac{n}{1-|\varphi^{\prime}(w)|}\big\rfloor-n}\text{.}$$
\par
Now suppose that  $\psi(w)\neq 0$ and $\phi(z)=\psi(z)\big(\frac{w-z}{1-\overline{w}z}\big)^{n}$. By replacing $\phi$ with $\psi$ in (\ref{e3.1}), we  have
$$\big\|D_{\psi,\varphi,n}\big\|_{-1}=\big\|D_{\phi,\varphi,n}\big\|_{-1} \geq \frac{n!| \psi(w)|}{\big(1-|w|^{2}\big)^{n}}
\binom{\big\lfloor \frac{n}{1-|\varphi^{\prime}(w)|}\big\rfloor}{n}
\big|\varphi^{\prime}(w)\big|^{\big\lfloor \frac{n}{1-|\varphi^{\prime}(w)|}\big\rfloor-n}\text{.}$$\par
Note that if $\varphi^{\prime}(w)=0$ and $\psi^{\prime \prime}(w)=0$, then  $D_{\psi,\varphi,1}^{\ast}K_{w,-1}^{[2]}=\overline{\psi(w) \varphi^{\prime \prime}(w)}K_{w,-1}^{[2]}$ by \cite[Lemma 1]{s2}. Therefore, we conclude that $\big\|D_{\psi,\varphi,1}\big\|_{-1} \geq \big|\psi(w) \varphi^{\prime \prime}(w)\big|$. Hence the result follows.
\end{proof}

In the next example, we show that for some operators $D_{\varphi}$, Proposition \ref{5theorem} is more useful than \cite[Proposition 4]{fh} for  estimating the lower bound for  $\big\|D_{\varphi}\big\|_{-1}$.

\begin{ex}\label{2example}
Suppose that $\varphi(z)=az^{n}+bz$, where $\frac{1}{2}< |b|<1-|a|$ and $n$ is a positive integer that $n\geq 2$. Proposition \ref{5theorem} implies that
$$\big\|D_{\varphi}\big\|_{-1} \geq  \bigg\lfloor \frac{1}{1-|b|} \bigg\rfloor |b|^{\lfloor 1 /(1-|b|)\rfloor -1}>1$$
and so this lower bound is greater than the lower bound for $\big\|D_{\varphi}\big\|_{-1}$ which was estimated in \cite[Proposition 4]{fh}.

\end{ex}

In the following proposition, we obtain $\|D_{\psi,\varphi,n}\|_{\alpha}$, when $D_{\psi,\varphi,n}$ is a compact hyponormal (or cohyponormal) operator which satisfies the hypotheses of Proposition \ref{1prop}.

\begin{prop}\label{6proposition}
Suppose that $\psi$ is not identically zero and $\varphi$ is a nonconstant analytic self-map of $\mathbb{D}$ so that $D_{\psi,\varphi,n}$ is compact on $\mathcal{H}_{\alpha}$. Assume that $w$ is the fixed point of $\varphi$ and $\psi$ has a zero at $w$ of order at least $n$. Then $D_{\psi,\varphi,n}$ is hyponormal or  cohyponormal on  $\mathcal{H}_{\alpha}$ if and only if $\psi(z)=az^{n}$ and $\varphi(z)=bz$, where $a \in \mathbb{C}\setminus \{0\}$ and $b \in \mathbb{D}\setminus \{0\}$; moreover, in this case
$$\big\|D_{\psi,\varphi,n}\big\|_{\alpha}=n!|a|\binom{\big\lfloor \frac{n}{1-|b|}\big\rfloor}{n}|b|^{\big\lfloor \frac{n}{1-|b|}\big\rfloor-n}\text{.}$$
\end{prop}

\begin{proof}
Suppose that $D_{\psi,\varphi,n}$ is hyponormal (or cohyponormal). If $\psi^{(n)}(w)=0$, then  $r_{\alpha}(D_{\psi,\varphi,n})=0$  by Theorem \ref{4theorem} and so $D_{\psi,\varphi,n}\equiv0$ by \cite[Proposition 4.6, p. 47]{c4}. It follows that $\psi \equiv 0$ or $\varphi \equiv 0$ (note that $D_{\psi,\varphi,n}\big(z^{n+1}\big)=(n+1)!\psi(z) \varphi(z)$) which is a contradiction. Hence we assume that $\psi^{(n)}(w)\neq0$. Since $\psi$ has a zero at $w$ of order $n$, Lemma \ref{lemma2} shows that $D_{\psi,\varphi,n}^{\ast}K_{w,\alpha}=0$. Hence $K_{w,\alpha}$ is an eigenvector for $D_{\psi,\varphi,n}^{\ast}$ corresponding to eigenvalue $0$. Since $D_{\psi,\varphi,n}$ is hyponormal (or cohyponormal), we conclude that $D_{\psi,\varphi,n} K_{w,\alpha}(z)=\frac{(\alpha+2)...(\alpha+n+1)\overline{w}^{n}\psi(z)}{(1-\overline{w}\varphi(z))^{\alpha+2+n}}=0$
 and so $w=0$. Lemma \ref{lemma2} implies that
 $$D_{\psi,\varphi,n}^{\ast}K_{0,\alpha}^{[n]}(z)=\overline{\psi^{(n)}(0)}K_{0,\alpha}^{[n]}(z)=\overline{\psi^{(n)}(0)}
(\alpha+2)...(\alpha+n+1)z^{n}\text{.}$$
 Since $D_{\psi,\varphi,n}$ is hyponormal (or cohyponormal), it follows that
 $$D_{\psi,\varphi,n}K_{0,\alpha}^{[n]}(z)=\psi^{(n)}(0)(\alpha+2)...(\alpha+n+1)z^{n}\text{.}$$
 Because
 $D_{\psi,\varphi,n}K_{0,\alpha}^{[n]}=n!(\alpha+2)...(\alpha+n+1)\psi$, we conclude that
 $\psi(z)=\frac{\psi^{(n)}(0)}{n!}z^{n}$, where $\psi^{(n)}(0)\neq 0$. Then  $\psi^{(m)}(0)=0$ for each $m\neq n$. Hence Lemma \ref{lemma2} shows that $$D_{\psi,\varphi,n}^{\ast}K_{0,\alpha}^{[n+1]}=(n+1)\overline{\psi^{(n)}(0)\varphi^{\prime}(0)}K_{0,\alpha}^{[n+1]}\text{.}$$
 Therefore, we have
 \begin{equation}\label{eI}
D_{\psi,\varphi,n}K_{0,\alpha}^{[n+1]}=(n+1)\psi^{(n)}(0)\varphi^{\prime}(0)K_{0,\alpha}^{[n+1]}\text{.}
 \end{equation}
 On the other hand, we obtain
 \begin{eqnarray} \label{eII}
 D_{\psi,\varphi,n}K_{0,\alpha}^{[n+1]}(z) &=& (n+1)!(\alpha+2)...(\alpha+n+2)\psi(z)\varphi(z) \nonumber \\
&=&(n+1)!(\alpha+2)...(\alpha+n+2)
 \frac{\psi^{(n)}(0)}{n!}z^{n}\varphi(z)
 \end{eqnarray}
 for each $z \in \mathbb{D}$. Since $D_{\psi,\varphi,n}$ is hyponormal (or cohyponormal), (\ref{eI}) and (\ref{eII}) imply that $\varphi(z)=\varphi^{\prime}(0)z$. \par
 Conversely is obvious by \cite[Proposition 3.2]{mfh} (note that an analogue of \cite[Proposition 3.2]{mfh} holds in $H^{2}$ by the similar idea). \par
 Due to the hyponormality (or cohyponormality) of $D_{\psi,\varphi,n}$, invoking Theorem \ref{20theorem},  it follows that
 $$\big\|D_{\psi,\varphi,n}\big\|_{\alpha}=r_{\alpha}(D_{\psi,\varphi,n})=n!|a|\binom{\big\lfloor \frac{n}{1-|b|}\big\rfloor}{n}|b|^{\big\lfloor \frac{n}{1-|b|}\big\rfloor-n}\text{.}$$
\end{proof}

In the next theorem, we extend \cite[Theorem 2]{fh}.

\begin{thm}\label{7theorem}
If $\varphi(z)=bz$ for some $b \in \mathbb{D}\setminus \{0\}$, then
$$\big\|D_{\varphi,n}\big\|_{-1}=n!\binom{\big\lfloor \frac{n}{1-|b|}\big\rfloor}{n}|b|^{\big\lfloor \frac{n}{1-|b|}\big\rfloor-n}\text{.}$$
\end{thm}

\begin{proof}
The result follows immediately from Proposition \ref{6proposition} and the fact that $T_{z^{n}}$ is an isometry on $H^{2}$.
\end{proof}

Let $\|\varphi\|_{\infty}\leq b<1$ and $\psi \in H^{\infty}$. We define $\varphi_{b}=(1/b)\varphi$ and $\rho(z)=bz$ (\cite[p. 2898]{fh}).
Observe that $D_{\varphi,n}=C_{\varphi_{b}}D_{\rho,n}$. Since $\big\|D_{\psi,\varphi,n}\big\|_{-1}\leq \|\psi\|_{\infty}\|C_{\varphi_{b}}\|_{-1}\|D_{\rho,n}\|_{-1}$, we can estimate the upper bound for $\big\|D_{\psi,\varphi,n}\big\|_{-1}$ by the same idea  as stated for the proof of  \cite[Proposition 4]{fh}.

\begin{prop}\label{8prop}
If $\varphi$ is a nonconstant analytic self-map of $\mathbb{D}$ with  $\|\varphi\|_{\infty}<1$ and the function $\psi$ belongs to $H^{\infty}$, then
$$\big\|D_{\psi,\varphi,n}\big\|_{-1}\leq n!\|\psi\|_{\infty}\sqrt{\frac{b+|\varphi(0)|}{b-|\varphi(0)|}}\binom{\big\lfloor \frac{n}{1-|b|}\big\rfloor}{n}|b|^{\big\lfloor \frac{n}{1-|b|}\big\rfloor-n}$$
whenever $\|\varphi\|_{\infty}\leq b<1$. In particular, $\big\|D_{\varphi,n}\big\|_{-1}=n!$ whenever both $\|\varphi\|_{\infty}\leq \frac{1}{n+1}$ and $\varphi(0)=0$.
\end{prop}


\begin{thebibliography}{9}


\bibitem{c4}\textsc{ J. \ B. \ Conway}, \textit{The Theory of Subnormal Operators}, Amer. Math. Soc., Providence, 1991.



\bibitem{cm} \textsc{C.\ C.\ Cowen and B.\ D.\ MacCluer}, \textit{Composition Operators on Spaces of Analytic Functions}, CRC Press, Boca Raton, 1995.

\bibitem{fh} \textsc{M.\ Fatehi and C.\ N.\ B.\ Hammond}, Composition--differentiation operators on the Hardy space, \textit{Proc.\ Amer.\ Math.\ Soc.}\ \textbf{148} (2020), 2893--2900.

\bibitem{s2} \textsc{M.\ Fatehi and C.\ N.\ B.\ Hammond}, Normality and self-adjointness of weighted composition-differentiation operators, \textit{Complex Anal. Oper. Theory}\ \textbf{15} (2021), 13.




\bibitem{hw} \textsc{K.\ Han and M.\ Wang}, Weighted composition--differentiation operators on the Bergman space, \textit{Complex Anal. Oper. Theory}\ \textbf{15} (2021), 17.


\bibitem{hp} \textsc{R.\ A.\ Hibschweiler and N.\ Portnoy}, Composition followed by differentiation between Bergman and Hardy spaces, \textit{Rocky Mountain. J.\ Math.}\ \textbf{35} (2005), 843--855.




\bibitem{mfh} \textsc{M.\ Moradi and M.\ Fatehi}, Complex symmetric weighted composition--differentiation operators of order $n$ on the weighted Bergman spaces, \textit{arXiv}:2101.04911.

\bibitem{moradi} \textsc{M.\ Moradi and M.\ Fatehi}, Products of composition and differentiation operators on the Hardy space, \textit{arXiv}.


\bibitem{ohno} \textsc{S.\ Ohno}, Products of composition and differentiation between Hardy spaces, \textit{Bull.\ Austral.\ Math.\ Soc.}\ \textbf{73} (2006), no.\ 2, 235--243.

\bibitem{richman} \textsc{A.\ E.\ Richman}, Subnormality and composition operators on the Bergman space, \textit{Integr. Equ. Oper. Theory} \textbf{45} (2003), 105--124.



\bibitem{s4} \textsc{S. Stevi\' c}, Products of composition and differentiation operators on the weighted Bergman space, \textit{Bull. Belg. Math.  Soc. Simon Stevin} \textbf{16} (2009), 623--635.
\end{thebibliography}
 \end{document}